\documentclass[10pt]{article}
\usepackage{amssymb, amsfonts, amsthm, mathrsfs, amsmath}
\usepackage[all,cmtip]{xy}
\usepackage{tipa}

\theoremstyle{definition}
\newtheorem{thm}{Theorem}[section]
\newtheorem{dfn}[thm]{Definition}
\newtheorem{lemma}[thm]{Lemma}

\newtheorem{prop}[thm]{Proposition}

\newtheorem{rmk}[thm]{Remark}

\def\id{\mathrm{id}}

\def\ccc{\mathbb{C}}

\def\zz{\mathbb{Z}}
\def\rr{\mathbb{R}}

\def\pp{\mathbb{P}}

\def\clo{\mathcal{O}}

\def\frI{\mathfrak{I}}

\def\pt{\partial}

\def\bpt{\bar{\pt}}
\def\ddb{\pt\bpt}

\def\ud{\mathrm{d}}

\def\bea{\begin{equation}}
\def\eea{\end{equation}}

\def\<{{\langle}}
\def\>{{\rangle}}

\begin{document}

\title{Remarks on Generalized Calabi-Gray Manifolds}
\author{Teng Fei}

\date{}

\maketitle{}

\begin{abstract}
Generalized Calabi-Gray manifolds are non-K\"ahler complex manifolds with very explicit geometry yet not being homogeneous. In this note, we demonstrate that how generalized Calabi-Gray manifolds can be used to answer some questions in non-K\"ahler geometry.
\end{abstract}

\section{Introduction}

Generalized Calabi-Gray manifolds are introduced in \cite{fei2016} as new examples of complex non-K\"ahler manifolds with very accessible geometry. Their constructions depend on a nonconstant holomorphic map $\varphi:Y\to\ccc\pp^1$ and a hyperk\"ahler manifold $M$. When $\varphi$ is the Gauss map of an immersed minimal surface $Y$ in a flat 3-torus and that $M$ is of real dimension 4, we recover the classical constructions due to Calabi \cite{calabi1958} and Gray \cite{gray1969}. Generalized Calabi-Gray manifolds have found their applications in solving the Hull-Strominger system \cite{fei2017b,fei2021,fei2020d}, the Anomaly flow \cite{fei2021d}, and building new models in theoretical physics \cite{chen2019} etc.

Noticeably, there has been an uprising interest in complex non-K\"ahler geometry in recent years. Many new ideas and conjectures have been proposed, and many exciting conferences are held. This particular piece of work is inspired by author's discussion with his colleagues at one of the \emph{Geometry \& TACoS} workshops.\footnote{For more information, see https://events.dimai.unifi.it/tacos/sessions/geometry-and-physics-of-non-kahler-calabi-yau/.} It turns out that generalized Calabi-Gray manifolds (with slight modification in some cases) can be used to answer many newly proposed questions in non-K\"ahler geometry. 

This paper is organized as follows. In Section 2, we first recall the construction of the generalized Calabi-Gray manifolds and their relation with twistor spaces. Then we derive some existence and non-existence results about special metrics on generalized Calabi-Gray manifolds. In Section 3 and 4, we demonstrate that generalized Calabi-Gray manifolds (with slight modification if necessary) provide counterexamples to two questions in complex non-K\"ahler geometry. Along the way, we also prove the existence of balanced metrics on suspension of balanced manifolds with invariant volume forms and present a brief list of interesting questions.

The author would like to thank all the participants of the \emph{Geometry \& TACoS} workshop for their inspiring discussions. The author would also like to thank the Simons Foundation for the generous funding, as this work is partially supported by the Simons Collaboration Grant 853806.

\section{The Generalized Calabi-Gray Manifolds}

Let us recall the construction of generalized Calabi-Gray manifolds from \cite{fei2016}, with some extensions to a broader setting. 

Let $M$ by a compact hyperk\"ahler manifold, namely that $M$ is equipped with a Riemannian metric $g$ and three integrable complex structures $I$, $J$, and $K$ satisfying the quaternion relation $IJK=-1$, such that all of $I$, $J$, and $K$ together with $g$ define K\"ahler metrics on $M$. In fact, one gets an $S^2=\ccc\pp^1$ worth of K\"ahler structures on $M$ as follows. Let $\omega_I$, $\omega_J$, and $\omega_K$ be the K\"ahler forms associated to $I$, $J$, and $K$ respectively. For any point $(a,b,c)\in S^2=\{(a,b,c)\in\rr^3:a^2+b^2+c^2=1\}$, the complex structure $aI+bJ+cK$ defines a K\"ahler metric on $M$, whose K\"ahler form is simply $a\omega_I+b\omega_J+c\omega_K$.

It is very natural to consider the twistor space $Z$ of a hyperk\"ahler manifold $M$ (see for example \cite{hitchin1987}), which encodes the whole $\ccc\pp^1$-family of K\"ahler structures on $M$. As a smooth manifold, the twistor space is simply the product $Z=M\times\ccc\pp^1$. However, $Z$ is equipped with a tautological yet twisted integrable complex structure $\frI$, which can be described as follows. For any point $z=(m,(a,b,c))\in M\times\ccc\pp^1=Z$, we have
\[\frI_z=(aI+bJ+cK)_m\oplus j_{(a,b,c)},\]
where $j$ is the canonical complex structure on $\ccc\pp^1$. It is well-known that one can recover all the information about $M$ from the complex geometry of $Z$. The twistor space $Z$ has the following properties.

\begin{prop}\label{twi1}(See for example \cite{hitchin1987})
\begin{enumerate}
\item The natural projection $\pi:Z=M\times\ccc\pp^1\to\ccc\pp^1$ is holomorphic.
\item For any point $m\in M$, $\{m\}\times\ccc\pp^1\subset M\times\ccc\pp^1=Z$ is a complex submanifold of $(Z,\frI)$.
\item The bundle $\bigwedge^2F^*_\pi\otimes\pi^*\clo(2)$ over $Z$ has a global holomorphic section which defines a holomorphic symplectic form over each fiber of $\pi$. Here $\clo(1)$ is the hyperplane bundle over $\ccc\pp^1$ and $F^*_\pi$ is the relative cotangent bundle of the holomorphic fibration $\pi:Z\to\ccc\pp^1$.
\end{enumerate}
\end{prop}

Let $2m$ be the complex dimension of the hyperk\"ahler manifold $M$, then we know that $Z$ is a complex manifold of complex dimension $2m+1$. Regarding Hermitian metrics on twistor spaces, we have the following well-known results.

\begin{prop}\label{twi2}~
\begin{enumerate}
\item $Z$ admits balanced metrics, which are Hermitian metrics such that their $2m$-th powers are closed.
\item $Z$ does not admit any K\"ahler metric. In fact more generally, $Z$ does not admit any astheno-K\"ahler metric.
\end{enumerate}
\end{prop}

\begin{proof}~
\begin{enumerate}
\item The natural Hermitian metric $\omega$ (the product metric) on $Z$ takes the form $\omega=\omega_{FS}+a\omega_I+b\omega_J+c\omega_K$, where $\omega_{FS}$ is the Fubini-Study metric on $\ccc\pp^1$. It is straightforward to check that $\ud(\omega^{2m})=0$.
\item Recall that a Hermitian metric $\omega$ on a complex $n$-fold is called astheno-K\"ahler if $i\ddb(\omega^{n-2})=0$. Clearly K\"ahler metrics are special cases of astheno-K\"ahler metrics. Suppose that $Z$ admits an astheno-K\"ahler metric $\omega_1$ satisfying $i\ddb(\omega_1^{2m-1})=0$.Then by integration by part, we get
    \begin{equation}
    0=\int_Zi\ddb(\omega_1^{2m-1})\wedge\omega=\int_Z\omega_1^{2m-1}\wedge i\ddb\omega.\label{ibp}
    \end{equation}
    On the other hand, using the computation given in \cite[Theorem 5.4]{fei2016} or \cite[Theorem 3.2]{deschamps2017}, we have
    \[i\ddb\omega=\omega_{FS}\wedge\omega,\]
    hence $\omega_1^{2m-1}\wedge i\ddb\omega$ is a positive top form, contradiction!
\end{enumerate}
\end{proof}

Now let $\varphi:Y\to\ccc\pp^1$ be a nonconstant holomorphic map from a compact complex manifold $Y$ to $\ccc\pp^1$, the base of the twistor family $\pi:Z\to\ccc\pp^1$. Using $\varphi$, one can pullback the twistor fibration $\pi$ to a holomorphic fibration $p:X\to Y$, fitting into the following commuting square:
\[\xymatrix{X=\varphi^*Z\ar[r]\ar[d]_p&Z\ar[d]^\pi\\ Y\ar[r]^\varphi&\ccc\pp^1}\]
In this note, we shall call such an $X$ a \emph{greatly generalized Calabi-Gray manifold}\footnote{Greatly generalized Calabi-Gray manifolds are called \emph{higher analogs of twistor spaces} in \cite{lin2017}.}. Just like the twistor space $Z$, the space $X$, as a smooth manifold, is simply the product $X=M\times Y$ with a twisted complex structure. It is named after Calabi and Gray because when $Y$ is a minimal surface in a flat $T^3$ and $\varphi$ its Gauss map, and that $M$ is a hyperk\"ahler 4-manifold like $T^4$ or a K3 surface, the corresponding $X$ recovers the classical construction of Calabi \cite{calabi1958} and Gray \cite{gray1969}. Greatly generalized Calabi-Gray manifold in this broad sense also include the examples studied by LeBrun \cite{lebrun1999} where $\varphi$ is a polynomial map from $\ccc\pp^1$ to itself.

Let $X$ be a greatly generalized Calabi-Gray manifold. Very much like Proposition \ref{twi1}, we have
\begin{prop}\label{gtwi1}~
\begin{enumerate}
\item The natural projection $p:X\to Y$ is holomorphic.
\item For any point $m\in M$, $\{m\}\times Y\subset M\times Y=X$ is a complex submanifold of $X$.
\item The canonical bundle $K_X$ of $X$ can be computed as
\begin{equation}\label{can}
K_X\cong p^*(K_Y\otimes \varphi^*\clo(-2m)).
\end{equation}
\end{enumerate}
\end{prop}

As a generalization of Proposition \ref{twi2}, the following statements hold.

\begin{prop}\label{gtwi2}(cf. \cite{lin2017})
\begin{enumerate}
\item If $Y$ admits a balanced metric, so does $X$.
\item $X$ cannot admit an astheno-K\"ahler and a plurisubharmonic metric at the same time. In particular, $X$ is always non-K\"ahler.
\end{enumerate}
\end{prop}
\begin{proof}~
\begin{enumerate}
\item Let $\omega_Y$ be any balanced metric on $Y$. It is straightforward to check that the product metric $\omega_Y+a\omega_I+b\omega_J+c\omega_K$ on $X=M\times Y$ is a balanced metric.
\item Recall that a Hermitian metric $\omega$ is called \emph{plurisubharmonic} if $i\ddb\omega\geq 0$ as a (2,2)-form. 

    First, we show that $X$ admits a plurisubharmonic metric if and only if $Y$ does so. If $X$ admits a plurisubharmonic metric, so do all complex submanifolds of $X$. In particular, by Proposition \ref{gtwi1}(b), $Y$ also admits a plurisubharmonic metric. On the other hand, let $\omega_2$ be a plurisubharmonic metric on $Y$, then the metric $\omega=\omega_2+a\omega_I+b\omega_J+c\omega_K$ is a plurisubharmonic metric on $X$. 
    
    Next, let us assume that there is also an astheno-K\"ahler metric $\omega_1$ on $X$. Exactly the same calculation as in (\ref{ibp}) yields the contradiction. Here we need that
    \[i\ddb\omega=i\ddb\omega_2+\varphi^*\omega_{FS}\wedge(a\omega_I+b\omega_J+c\omega_K)\geq \varphi^*\omega_{FS}\wedge(a\omega_I+b\omega_J+c\omega_K),\]
    and that $\varphi$ is not a constant so $\omega_1^{2m-1}\wedge\varphi^*\omega_{FS}\wedge(a\omega_I+b\omega_J+c\omega_K)$ is strictly positive on an open subset of $X$.
\end{enumerate}
\end{proof}
\begin{rmk}
Complex manifolds admitting both an astheno-K\"ahler metric and a plurisubharmonic metric form a strictly larger category than K\"ahler manifolds. In fact, there exists a compact complex non-K\"ahler manifold of complex dimension $n$ with a Hermitian metric $\omega$ satisfying $i\ddb(\omega^k)=0$ for $k=1,2,\dots,n-1$. Such metrics are studied in \cite{fino2011,shen2022}.
\end{rmk}

\begin{dfn}
A \emph{generalized Calabi-Gray manifold} $X$ is a greatly generalized Calabi-Gray manifold with the extra condition that
\begin{equation}
K_Y\cong\varphi^*\clo(2m).\label{gcg}
\end{equation}
In view of (\ref{can}), this condition holds if and only if the canonical bundle of $X$ is holomorphically trivial, in which case $X$ is a non-K\"ahler Calabi-Yau manifold.
\end{dfn}

\section{The First Question}

Let $X$ be a compact K\"ahler manifold. It is well-known that if the first Chern class of $X$ vanishes in the real de Rham cohomology group, then its canonical bundle $K_X$ must be torsion, meaning that a finite power of $K_X$ is holomorphically trivial. It is natural to ask if a similar statement holds when $X$ is non-K\"ahler, where we may work with a stronger assumption that the first Chern class of $X$ vanishes in the Bott-Chern cohomology instead. It is noted in \cite{tosatti2015} that Magn\'usson \cite{magnusson2017} gives an example of a compact non-K\"ahler manifold with vanishing first Chern class in the Bott-Chern cohomology whose canonical bundle has infinite order. 

As completely dropping the K\"ahler condition leads to Magn\'usson's counterexample \cite{magnusson2017}, one may ask what happens if we replace the K\"ahler condition by a slightly weaker one. In this spirit, Alexandra Otiman asks during the \emph{Geometry \& TACoS} workshop that if there exists a compact complex manifold with balanced metrics such that its first Chern class vanishes in the Bott-Chern cohomology while its canonical bundle is not torsion.

It turns out such an example exists by considering greatly generalized Calabi-Gray manifolds. To be more specific, we show that
\begin{prop}\label{construction}~\\
For a generic compact Riemann surface of genus $g\geq 5$, there exists a nonconstant holomorphic map $\varphi:Y\to\ccc\pp^1$, such that the associated greatly generalized Calabi-Gray manifold $X$ with $M$ being a hyperk\"ahler 4-manifold satisfies all the conditions in Otiman's question, namely:
\begin{enumerate}
\item $X$ admits a balanced metrics.
\item $c_1(X)=0\in H^{1,1}_{BC}(X;\rr)$.
\item $K_X$ has infinite order.
\end{enumerate}
\end{prop}

Because $Y$ is a Riemann surface, we know that Proposition \ref{construction}(a) holds automatically due to Proposition \ref{gtwi2}(a). For Proposition \ref{construction}(b), we show that it is essentially a topological condition:
\begin{lemma}~\\
Let $\varphi:Y\to\ccc\pp^1$ be a holomorphic map from a compact Riemann surface of genus $g$ to $\ccc\pp^1$ and let $M$ be a hyperk\"ahler 4-manifold. Proposition \ref{construction}(b) holds if and only if the topological degree of $\varphi$ is $g-1$. 
\end{lemma}
\begin{proof}
As the Bott-Chern class of $c_1(X)$ is represented by the curvature of $-K_X$, from (\ref{can}) we know that
\[c_1(X)\in p^*H^{1,1}_{BC}(Y;\rr).\] 
Because $Y$ is a compact Riemann surface, we also know that $H^{1,1}_{BC}(Y;\rr)\cong H^2(Y;\zz)\otimes\rr$, therefore $c_1(X)=0\in H^{1,1}_{BC}(X;\rr)$ if and only if $K_Y\cong\varphi^*\clo(2m)$ topologically over $Y$. Note that $m=1$ since $M$ is a hyperk\"ahler 4-manifold, we see immediately that $K_Y\cong\varphi^*\clo(2)$ topologically if and only if $\varphi$ is of degree $g-1$.
\end{proof}

Consequently, to prove Proposition \ref{construction}, we only need to show:
\begin{prop}~\\
For a generic compact Riemann surface of genus $g\geq 5$, there exists a holomorphic map $\varphi:Y\to\ccc\pp^1$ of degree $g-1$ such that $K_Y\otimes\varphi^*\clo(-2)$ is of infinite order in $Pic^0(Y)$, the Picard variety of $Y$.
\end{prop}
\begin{proof}
The proof is based on the following facts from algebraic geometry.
\begin{enumerate}
\item \cite[p.206]{arbarello1985} For any compact Riemann surface $Y$ of genus $g\geq 5$, every component of $W_{g-1}^1(Y)$ has dimension at least $g-4\geq 1$. Here $W_d^r(Y)$ is the variety that parameterizes all holomorphic line bundles of degree $d$ with $h^0\geq r+1$ on $Y$.    
\item \cite[pp.372-373]{arbarello1985} For a generic compact Riemann surface $Y$ of genus $g\geq 5$, a generic line bundle in $W_{g-1}^1(Y)$ is base-point-free.
\end{enumerate}
Consider a generic compact Riemann surface $Y$ of genus $g\geq 5$, we take a generic line bundle $L\in W_{g-1}^1(Y)$. Since the linear series associated to $L$ is base-point-free, the sections of $L$ give rise to a holomorphic map $\varphi:Y\to\ccc\pp^1$ such that $\varphi^*\clo(1)\cong L$. In particular, $\varphi$ is of degree $g-1$ and $K_Y\otimes L^{-2}\in Pic^0(Y)$. As torsion elements in $Pic^0(Y)$ are countable and the dimension of $W_{g-1}^1(Y)$ is at least 1, we see that a generic choice of $L$ would make $K_Y\otimes L^{-2}$ of infinite order in $Pic^0(Y)$.
\end{proof}

A remaining question is that does Magn\'usson's example \cite{magnusson2017} carry a balanced metric? Magn\'usson's example falls in the broader construction called ``suspension'' that we shall describe. Let $M$ be a complex manifold and $f$ a holomorphic automorphism of $M$. Let $\tau$ be a complex number with positive imaginary part. We can define a $\zz^2$-action on $M\times\ccc$ by
\[(1,0)\cdot(x,z)=(f(x),z+1),\quad (0,1)\cdot(x,z)=(x,z+\tau).\]
It is clear that this $\zz^2$-action is free, so the quotient manifold, denoted by $M_f$, is smooth and equipped with a projection map over $C=\ccc/(\zz+\zz\tau)$.
Moreover, the $\zz^2$-action preserves the natural complex structure on $M\times\ccc$ so the complex structure descends to $X$ and the projection map $\pi:M_f\to C$ is a holomorphic fibration. Such a suspension construction has been recently studied in Qin-Wang \cite{qin2018} and Fino-Grantcharov-Verbitsky \cite{fino2022}.

Regarding the existence of balanced metrics on $M_f$, we have the following:
\begin{prop}\label{sus}~\\
Suppose $M$ admits a balanced metric whose volume form is invariant under $f$, then $M_f$ admits balanced metrics.
\end{prop}

In the case of Magn\'usson's example, $M$ is a K\"ahler manifold with holomorphically trivial canonical bundle, therefore any Ricci-flat K\"ahler metric on $M$ satisfies the condition in Proposition \ref{sus}. As a corollary, we conclude that Magn\'usson's example provides an alternative manifold satisfying all conditions asked by Otiman.

It is clear that as a smooth manifold, we have $M_f\cong M'_f\times S^1$, where $M'_f$ is the quotient of $M\times\rr$ by the $\zz$-action
\[1\cdot(x,z)=(f(x),z+1)\]
coming together with a projection $\pi':M'_f\to S^1$. It follows that $\pi:M_f\to C$ decomposes as $\pi=(\pi',\id)$ under the obvious identifications $M_f\cong M'_f\times S^1$ and $C\cong S^1\times S^1$. Let $m$ be the complex dimension of $M$, we have the following useful lemma. 
\begin{lemma}~\\
Under the assumption of Proposition \ref{sus}, for any small open interval $I\subset S^1$, there exists a closed non-negative $(m,m)$-form on $M_f$ such that it is strictly positive away from $\pi'^{-1}(I)\times S^1\subset M'_f\times S^1=M_f$.
\end{lemma}
\begin{proof}
Without loss of generality, we may think of $S^1$ as the closed interval $[0,1]$ with two end points glued together and that $0\in I$. In this way, we may think of $M'_f$ as the quotient of $M\times[0,1]$ by the relation $(x,0)\sim(f(x),1)$. Let $(s,t)$ be the standard coordinates on $S^1\times S^1=C$, and let $\rho(s)$ be a smooth nonnegative function on $S^1$ vanishing in $I$.

Let $\omega$ be a balanced metric on $M$ such that $\omega^m$ is fixed by $f$. Consider the following nonnegative $(m,m)$-form
\[\Omega=\omega^m+\rho(s)\omega^{m-1}\ud s\wedge\ud t\]
on $M\times [0,1]\times S^1$. Since $\omega^m$ is invariant under $f$ and $\rho(0)=\rho(1)=0$, the above form descends to a non-negative $(m,m)$-form on $M_f$, with the desired property in the lemma. We need $\omega$ to be a balanced metric on $M$ to check that $\ud\Omega=0$.
\end{proof}

To prove Proposition \ref{sus}, we simply take two non-intersecting open intervals $I_1$ and $I_2$ in $S^1$ and apply the above Lemma to construct $\Omega_1$ and $\Omega_2$ respectively. Then $\Omega=\Omega_1+\Omega_2$ is a closed positive $(m,m)$ form on $M_f$, which is equivalent to a balanced metric. This construction gives a constructive realization of a special case in \cite[Theorem 5.5]{michelsohn1982} without invoking verification of the ``topologically essential'' condition.

\section{The Second Question and More}

In his talk ``Non-K\"ahler manifolds with vanishing first Bott-Chern class''\footnote{For the video, see https://www.youtube.com/watch?v=7911FB5alFg.} at \emph{Geometry \& TACoS}, Tosatti asks the following question:\\ 

\noindent\textbf{Question.} Let $X$ be a compact complex manifold with vanishing first Chern class in the Bott-Chern cohomology. Is it true that holomorphic tensors on $X$ must be parallel with respect to Chern connection of a Chern-Ricci flat metric?\\

The motivation of this question is that its K\"ahler version holds by applying the Bochner technique. By considering generalized Calabi-Gray manifolds, we see the answer to this question is negative. In particular, we have:

\begin{prop}~\\
There exist generalized Calabi-Gray manifolds with Chern-Ricci flat balanced metrics such that its nonzero holomorphic 1-forms (which exist in abundance) are not parallel under any connection. 
\end{prop}
\begin{proof}
In fact, we can consider the simplest kind of generalized Calabi-Gray manifolds introduced in \cite{fei2016,fei2021}. These are generalized Calabi-Gray manifolds associated to $\varphi:Y\to\ccc\pp^1$, where $Y$ is a Riemann surface of genus $g$, and that the hyperk\"ahler factor $M$ is either a K3 surface or a flat 4-torus. The canonical bundle $K_X$ of $X$ is holomorphically trivial as $\varphi$ satisfies condition (\ref{gcg}) that
\[K_Y\cong\varphi^*\clo(2).\]
It is shown in \cite{fei2016} that such a construction exists for any $g\geq 3$ and that the natural Hermitian metric on $X$ is balanced and Chern-Ricci-flat. Moreover, it is also shown that all holomorphic 1-forms on $X$ are pullbacks of holomorphic 1-forms on $Y$, which forms a complex vector space of dimension $g$. Since these are abelian differentials on a Riemann surface of high genus, they all have zeroes and cannot be parallel under any connection.
\end{proof}

Let me end this paper by proposing and formulating some questions about (greatly) generalized Calabi-Gray manifolds of interest.

\begin{enumerate}
\item How to compute all the Hodge numbers of a generalized Calabi-Gray manifold? Some of them are discussed in \cite{fei2016}. This question is highly related to computation of Hodge numbers of a twistor space, which is discussed in \cite{eastwood1993}.
\item This is a question asked my Misha Verbitsky. Consider the greatly generalized Calabi-Gray manifold $X$ associated to $\varphi:T^2\to\ccc\pp^1$ with $M=T^4$. In this case, $X$ is topologically a 6-torus with a twisted complex structure such that its first Chern class is nonpositive but not zero. Is it true that any complex structure on $T^6$ with zero first Chern class must be K\"ahler?
\end{enumerate} 
 
\bibliographystyle{alpha}

\bibliography{C:/Users/benja/Dropbox/Documents/Source}

\end{document}